\documentclass[a4paper]{amsart}

\usepackage{cite}
\usepackage{ upgreek }

\usepackage{amsmath, amsfonts, amssymb, amsthm, txfonts, color, xy, amsbsy, mathtools}
\input xy
\xyoption{all}

\usepackage[draft]{graphicx}
\usepackage{color}

\usepackage{chngcntr}
\counterwithin{figure}{section}

\newtheorem{theorem}{Theorem}[section]

\newtheorem{corollary}[theorem] {Corollary}
\newtheorem{definition}[theorem]{Definition}

\newtheorem{lemma} [theorem]{Lemma}

\newtheorem{remark}[theorem]{Remark}

\numberwithin{equation}{section}

\newcommand{\C}{\mathbb{C}}
\newcommand{\G}{{\mathbf{G}}}
\newcommand{\K}{\textsf{K}}
\newcommand{\g}{ \mathfrak{g}}

\newcommand{\so}{\mathbf I(\sigma)}
\renewcommand{\r}{\mathbb S}

\newcommand{\GL}{{\rm GL}}
\newcommand{\SL}{{\rm SL}}
\newcommand{\Hom}{{\rm Hom}}

\newcommand{\Ind}{{\rm Ind}}

\newcommand{\Irr}{{\rm Irr}}

\newcommand{\Ol}{\mathfrak{o}_{\ell}}
\newcommand{\OO}{\mathfrak{o}}
\newcommand{\sfK}{\mathsf{K}}
\newcommand{\cO}{\mathfrak{o}}

\newcommand{\tC}{\mathrm{C}}
\newcommand{\lri}{{\cO}}
\newcommand{\xx}{{\bar{x}}}
\newcommand{\kk}{\mathbf k}

\renewcommand{\sl}{\mathrm{sl}}

\renewcommand{\wp}{\mathfrak{p}}

\title[A multiplicity one theorem ]{A multiplicity one theorem for groups of \\ type $A_n$ over discrete valuation rings}

\author{Shiv Prakash Patel}
\address{SPP : Department of Mathematics\\ Indian Institute of Technology Delhi \\ Hauz Khas \\ New Delhi - 110016, INDIA.}
\email{shivprakashpatel@gmail.com}

\author{Pooja Singla}
\address{PS : Department of Mathematics\\ Indian Institute of Science (IISc) \\ Bangalore - 560012 \\ INDIA.}
\email{pooja@iisc.ac.in}

\date{\today}


\begin{document}

\begin{abstract}

Let $\cO$ be the ring of integers of a non-archimedean local field with the maximal ideal $\wp$ and the finite residue field of characteristic $p.$ Let $\G$ be the General Linear or Special Linear group with entries from the finite quotients $\cO/\wp^\ell$ of $\cO$ and $\mathbf{U}$ be the subgroup of $\G$ consisting of upper triangular unipotent matrices. We prove that the induced representation $\Ind^{\G}_{\mathbf{U}}(\theta)$ of $\G$ obtained from a {\it non-degenerate} character $\theta$ of $\mathbf{U}$ is multiplicity free for all $\ell \geq 2.$ This is analogous to the multiplicity one theorem regarding Gelfand-Graev representation for the finite Chevalley groups. We prove that for many cases the regular representations of $\G$ are characterized by the property that these are the constituents of the induced representation $\Ind^{\G}_{\mathbf{U}}(\theta)$ for some non-degenerate character $\theta$ of $\mathbf{U}$. We use this to prove that the restriction of a regular representation of General Linear groups over $\cO/\wp^\ell$ to the Special Linear groups is multiplicity free for all $\ell \geq 2$ and also obtain the corresponding branching rules in many cases.
\end{abstract}

\maketitle

\section{Introduction}
\label{sec: introduction}

Let $\mathbf F$ be a non-archimedean local field and $\OO$ be the ring of integers of $\mathbf F$ such that the residue field $\mathbf k$ is finite of characteristic $p$ and $| \mathbf{k} | =q.$
Let $\mathfrak{p}$ be the unique maximal ideal of $\OO$ and $\varpi$ be a fixed generator of $\mathfrak{p}.$
Let $\G$ be a split reductive  group scheme defined over $\cO$ and $\G(\cO)$ be the set of $\OO$-points of $\G.$ For $\ell \geq 1,$ let $\OO_{\ell} = \OO/ (\varpi^{\ell}).$ The groups $\G(\cO)$ are pro-finite therefore every complex continuous finite dimensional irreducible representation of $\G(\cO)$ factors through the quotient $\G(\cO_\ell)$ for some $\ell.$ 
So it suffices to consider the finite dimensional complex representations of the groups $\G(\cO_\ell)$. 
Let $\mathbf{U}$ be the unipotent radical of a Borel subgroup $\mathbf{B}$ of $\G.$ Fix $\theta : \mathbf{U}(\cO_\ell) \rightarrow \C^{\times}$  a one dimensional representation of $\mathbf{U}(\cO_\ell).$ 

An irreducible representation $\pi$ of the group $\G(\cO_\ell)$ is said to admit a $\theta$-Whittaker model if the dimension of the intertwiner space $\Hom_{\G(\cO_\ell)} \left(\pi, \Ind_{\mathbf{U}(\cO_\ell)}^{\G(\cO_\ell)} (\theta) \right)$ is non-trivial and  $\pi$ is said to admits at most one $\theta$-Whittaker model if the intertwiner space $\Hom_{\G(\cO_\ell)} \left(\pi, \Ind_{\mathbf{U}(\cO_\ell)}^{\G(\cO_\ell)} (\theta) \right)$ has dimension at most one. In view of this, the representation $\Ind_{\mathbf{U}(\cO_\ell)}^{\G(\cO_\ell)} (\theta)$ is multiplicity free if and only if every irreducible representation of $\G(\cO_\ell)$ admits at most one $\theta$-Whittaker model. It is natural to ask which irreducible representations of $\G(\cO_\ell) $ admit a $\theta$-Whittaker model? More specifically, to determine the dimension of the intertwiner space $\Hom_{\G(\cO_\ell)} \left(\pi, \Ind_{\mathbf{U}(\cO_\ell)}^{\G(\cO_\ell)} (\theta)\right)$ for any irreducible representation $\pi$ of $\G(\cO_\ell)$.

These questions have already appeared in literature in many contexts. For example, for $\ell = 1$ and $\G(\mathbb F_q) = \GL_n(\mathbb F_q),$  one knows from the work of Gelfand and Graev~\cite{MR0148765, MR0272916} that for a non-degenerate character $\theta$ of $\mathbf{U}(\mathbb F_q),$ the induced representation $\mathrm{Ind}_{\mathrm{U}(\mathbb F_q)}^{\GL_n(\mathbb F_q)}(\theta)$ (now a days also called Gelfand-Graev representation)  is multiplicity free and further every cuspidal representation of $\GL_n(\mathbb F_q)$ has a $\theta$-Whittaker model for any non-degenerate character $\theta$ of $\mathbf{U}(\mathbb F_q)$.  They generalized these results further for finite Chevalley groups. These results turn out to be very important because the cuspidal representations are known to be the building blocks for all irreducible ones in these cases via the Harish-Chandra induction so in a way the essential irreducible representations are captured by the Gelfand-Graev representation. For $\GL_n(\mathbf F),$ where $\mathbf F$ is a local field, Shalika~\cite[Theorem 1.6, 2.1]{Shalika74} proved that every (smooth) irreducible representation has at most one $\theta$-Whittaker model for non-degenerate $\theta$ and generalized this to split and quasi-split groups as well. The case of local fields has several applications, for example certain integrals involving functions in the Whittaker model gives rise to the local $L$-function \cite[Theorem 2.18]{JL1970} and these local $L$-functions are then used to define global $L$-functions \cite{JL1970, Bump97}. Our focus is on the corresponding questions for the groups $\G(\cO_\ell),$ for $\G = \GL_n$ or $ \SL_n.$ It is to be noted that the construction of all irreducible representations of $\G(\cO_\ell)$ are not yet known even for $\mathrm{Char}(\cO) = 0$ case, see~\cite{AKOV1, AKOV2} for few results in $\mathrm{Char}(\cO) = 0$ case.

Let $\G = \GL_n$ or $\SL_n$ and $\mathbf{U}(\Ol)$ be the group of upper triangular unipotent matrices in $\G(\Ol).$ A non-trivial one dimensional representation $\varphi : \Ol \rightarrow \C^{\times}$ such that $\varphi|_{\varpi^{\ell -1} \Ol} \neq 1$ is called a primitive character of $\cO_\ell.$ For any $n-1$ primitive characters $\varphi_1, \varphi_2, \ldots, \varphi_{n-1}$ of $\cO_\ell$, define a one dimensional representation $\theta_{(\varphi_1, \varphi_2, \ldots, \varphi_{n-1})} : \mathbf{U}(\cO_\ell) \rightarrow \C^{\times}$ by
\begin{equation}
\theta_{(\varphi_1,\varphi_2 , \ldots, \varphi_{n-1})}((x_{ij})) := \varphi_1 (x_{12}) \varphi_2 (x_{23}) \cdots  \varphi_{n-1}(x_{(n-1)n}).
\end{equation}

Any one dimensional representation of $\mathbf{U}(\cO_\ell)$ of the form $\theta_{(\varphi_1, \varphi_2, \ldots, \varphi_{n-1})}$ corresponding to the primitive characters $\varphi_1, \varphi_2, \ldots, \varphi_{n-1}$  of $\cO_\ell$ will be called a ``non-degenerate character" of $\mathbf{U}(\cO_\ell).$ This is a natural generalization of the well known notion of non-degenerate character of $\mathbf{U}(\mathbb F_q)$ for $\mathbf{U}(\cO_\ell)$. For a non-degenerate character $\theta$ of $\mathbf{U}(\cO_\ell),$  we consider the $\G(\cO_\ell)$-representation space $\mathrm{Ind}_{\mathbf{U}(\cO_\ell)}^{\G(\cO_\ell)}(\theta)$ and prove the following. 
\begin{theorem}
\label{thm:main-theorem}
For $\G = \GL_n$ or $\SL_n,$ the $\G(\cO_\ell)$-representation space  $\mathrm{Ind}_{\mathbf{U}(\cO_\ell)}^{\G(\cO_\ell)}(\theta)$ is multiplicity free for every non-degenerate character $\theta$ of $\mathbf{U}(\cO_\ell).$ 
\end{theorem}

The proof of this theorem is included in Section~\ref{sec:proof-of-multiplicity-one}. 
The above questions for the groups $\GL_n (\OO_{\ell})$ with $\ell \geq 2$ have already appeared in the work of Hill~\cite[Proposition~5.7]{MR1334228} where he proved that the space $\mathrm{Ind}_{\mathbf{U}(\cO_\ell)}^{\GL_n(\cO_\ell)}(\theta)$ is multiplicity free for even $\ell$ and for $\ell$ odd Hill~\cite[Proposition~5.7]{MR1334228} managed to prove that certain class of regular representations of $\GL_n(\cO_\ell)$ admits at most one $\theta$-Whittaker model.  We briefly recall the definition of regular representations of $\G(\cO_\ell),$ for more details, see Section~\ref{sec:regular-representation-construction} and also \cite{MR3737836}. We will define and use the term regular representations of $\SL_n(\cO_\ell)$ only under the conditions $(p, 2) = (p, n) = 1$. In general the definition of regular representations for $\SL_n(\cO_\ell)$ is slightly complicated than the Definition~\ref{def:regular-over-field} and we will avoid the general one in this article. 
 
 Let $\G(\cO_\ell)$ be either $\SL_{n}(\cO_\ell)$ with $(p, 2) = (p, n) = 1$ or $\GL_{n}(\cO_\ell).$ Let $\g$ denote the Lie algebra scheme of $\G$. For every positive integer $r,$ let $\K_{\ell}^r= \mathrm{Ker}(\G(\cO_\ell) \mapsto \G(\cO/\wp^r))$ be the $r$-th principal congruence subgroup of $\G(\cO_\ell).$ There are $\G(\kk)$-equivariant isomorphisms $\K_\ell^r/\K_\ell^{r+1} \cong (\g(\kk), +)$ and a $\G(\kk)$-equivariant isomorphism $x \mapsto  \varphi_x$ between $\g(\kk)$ and its Pontryagin dual $\g(\kk)^\vee= \mathrm{Hom}_{\mathbb Z}(\g(\kk), \mathbb C^\times)$ (see \cite[Section~2.3]{MR3737836}). 
An element $x \in \g(\kk)$ is called regular or cyclic  if the characteristic polynomial of $x$ is equal to its minimal polynomial. An element $x \in \g(\kk)$ is called cuspidal if the characteristic polynomial of $x$ is irreducible. A one dimensional representation $ \varphi_x \in \g(\kk)^\vee$ is called regular (cuspidal) if $x$ is a regular (cuspidal) element.  
\begin{definition}\label{def:regular-over-field} [Regular (Cuspidal) representations]
 An irreducible representation $\rho$ of $\G(\cO_\ell)$ is called regular (cuspidal)  if the orbit of its restriction to $\K_\ell^{\ell-1}  \cong \g(\kk)$ consists of regular (cuspidal) one dimensional representations. 
\end{definition} 
It is to be noted that every cuspidal representation of $\G(\cO_\ell)$ is regular but the converse is not true for any $n \geq 2$. A construction of all cuspidal representations of $\GL_n(\cO_\ell)$ was given by Aubert et al.~\cite{MR2607551}. 
For odd $p$ and $\ell \geq 2$, the construction of all regular representations for $\SL_{n}(\OO_{\ell})$ with $p \nmid n$ and for all $\GL_{n}(\OO_{\ell})$ was obtained in \cite{MR3737836}. For general $p,$ an independent construction of regular representations of $\GL_{n}(\OO_{\ell})$ is given by Stasinski-Stevens~\cite{MR3743488}. 
The regular representations for ''good" characteristic of the residue field, have also been constructed when $\G$ is a unitary, orthogonal or symplectic group, see \cite{MR3737836, Shai-regular}. The case of $\G = \SL_n$ for $p \mid n$ is not yet completely known even for $n=2,$ see~\cite{Hassain-Singla-sl2} for few results regarding this. In this article, we complete Hill's results for $\GL_n(\cO_\ell)$ with $\ell$ odd as well as generalise these to $\SL_n(\cO_\ell)$ for $p \nmid n.$  In particular, we prove the following. 
\begin{theorem}
	\label{main-theorem-2}
	Let $\ell \geq 2$, and $\G(\cO_\ell)$ be either $\SL_{n}(\cO_\ell)$ with $(p, 2) = (p, n) = 1$ or $\GL_{n}(\cO_\ell).$ Then the following are true. 
	\begin{enumerate} 
		\item An irreducible representation of $\G(\Ol)$ admits a $\theta$-Whittaker model for some non-degenerate character $\theta$ of $\mathbf U(\cO_\ell)$ if and only if it is regular. 
        \item For every non-degenerate character $\theta$ of $\mathbf U(\cO_\ell)$, every irreducible representation of $\G(\Ol)$ admits at most one $\theta$-Whittaker model.
	\end{enumerate}
\end{theorem}

This theorem follows from a more general result, see Theorem~\ref{thm: existence}, and its proof is included in Section~\ref{sec:proof-of-whittaker-model}. The proofs of these results depend on the description and the construction of the regular representations as given in \cite{MR3737836, MR3743488}, see also Section~\ref{sec:regular-representation-construction} where we recall briefly this construction. The proof of the second part of Theorem~\ref{main-theorem-2} is quite indirect. Usually for a proof of multiplicity one theorem for a $G$-representation $V = \mathrm{Ind}_H^G(\delta)$, where $H $ is a subgroup of $G$ and $\delta$ is a representation of $H$, one expects to use the Gelfand's trick.  This trick relies on defining a good involution on $G$ which in turn fixes all the functions $f \in \mathrm{End}_G(V)$, where $f$ is considered as a function on the double cosets $H\backslash G /H$. Therefore to use this trick, one requires a good description of double cosets $\mathbf{U}(\Ol) \backslash \G(\Ol) /\mathbf{U}(\Ol).$ In the present case such a description is known only for $\G = \GL_2(\cO_\ell)$ and $\GL_3(\cO_\ell)$ (see \cite{MR2267575}) and is non-trivial even for $\GL_3(\cO_\ell)$. The description of double cosets $\mathbf{U}(\Ol) \backslash \G(\Ol) /\mathbf{U}(\Ol)$ is not known in general and is believed to be quite complicated. Therefore our proof completely avoids the approach using Gelfand's trick and is more towards generalising the ideas of Hill. For the proof, we count the sum of dimensions of few specific regular representations of $\G(\cO_\ell)$. This is achieved by carefully following the steps in the construction of the regular representations. In our case, it turns out that this sum is equal to the dimension of $\Ind_{\mathbf{U}(\Ol)}^{\G(\Ol)}(\theta).$ 

As a corollary of the proof of Theorem~\ref{main-theorem-2}, we obtain the following result of Hill\cite[Proposition~5.7]{MR1334228}, see the discussion preceding Theorem~\ref{thm: existence}.  

\begin{corollary}
\label{cor:hill-regular-iff-whittaker}
For $\G(\cO_\ell) = \GL_n(\cO_\ell)$ and a non-degenerate character $\theta$ of $\mathbf U(\cO_\ell)$, an irreducible representation of $\G(\Ol)$ admits a $\theta$-Whittaker model if and only if it is regular. 
\end{corollary}
The above result is not true for $\SL_n(\cO_\ell)$, see Sections~\ref{sec:examples} and \ref{sec:branching-rules}. We characterize the regular representations of $\SL_{n}(\cO_\ell)$ for $(p, 2) = (p, n) = 1$ that admit a $\theta$-Whittaker model for every non-degenerate character $\theta$ of $\mathbf U(\cO_\ell)$, see Theorem~\ref{prop:special-regular-characterization}. From Theorem~\ref{thm:main-theorem}, we also obtain the following result. 

\begin{theorem}
\label{thm:multiplicity-free-induction}
Let $\rho$ be a regular representation of $\GL_n(\cO_\ell)$ then the restriction of $\rho$ to  $\SL_n(\cO_\ell)$  is multiplicity free. 
\end{theorem}
A proof of this result is included in Section~\ref{sec:branching-rules}. This result is quite interesting in view of the fact that although the construction of regular representations of $\GL_n(\cO_\ell)$ is known in general but very less is known regarding the irreducible representations of $\SL_n(\cO_\ell)$ for $p \mid n$. The above result may be useful in finding more information regarding the irreducible representations of $\SL_n(\cO_\ell)$ for $p \mid n$. 

For $(p,2) = (p,n) =1$, Theorem~\ref{thm:multiplicity-free-induction} along with the information of dimensions of regular representations of $\GL_n(\cO_\ell)$ and $\SL_n(\cO_\ell)$ as given in \cite{MR3737836}, we determine the explicit dimension of the endomorphism algebra $\mathrm{End}_{\SL_n(\cO_\ell)} \left( \mathrm{Res}_{\SL_n(\cO_\ell)}^{\GL_n(\cO_\ell)} (\rho)  \right)$ for any regular representation $\rho$ of $\GL_n(\cO_\ell)$, see Theorem~\ref{thm:branching-numbers}. We believe that the results of Theorems~\ref{thm:multiplicity-free-induction} and ~\ref{thm:branching-numbers} will prove to be helpful in describing the branching rules for the supercuspidal representations of $\SL_n(\mathbf F)$ as was achieved by  Nevins~\cite{MR3008903} for $\SL_2(\mathbf F)$.

\subsection{Notations} We include here few notations that we use throughout this article. First, it is to be noted that $\kk = \cO_1 = \mathbb F_q,$ we use these interchangeably depending on the context and the convenience. For a group $G,$ the set of all equivalence classes of irreducible representations of $G$ is denoted by $\mathrm{Irr}(G).$ The set of all one dimensional representation of an abelian group $A$ is also denoted by $\widehat{A}.$ 
For a normal subgroup $N$ of $G$ and an irreducible representation $\varphi$ of $N,$ the set of all $\rho \in \mathrm{Irr}(G) $ such that the restriction $\rho|_N$ of $\rho$ to $N$ has $\varphi$ as a non-trivial constituent is denoted by $\mathrm{Irr}(G \mid \varphi)$ and any element of $\mathrm{Irr}(G \mid \varphi)$ is at times called a representation of $G$ lying above $\varphi$.

\section{Regular elements of $\g(\cO_r)$ }
\label{sec:regular-elements}
Throughout this section, we assume that either $\G = \SL_n$ with $(p,n) = (p,2) = 1$ or $\G = \GL_n$. Let $\g$ denote the Lie algebra scheme of $\G.$ In this section, we collect some facts regarding the regular elements of $\g(\cO_r)$ for $r \geq 1$.

An element $x \in \g(\cO_r) \subseteq M_n(\cO_r)$ is called {\it  cyclic} if there exists an element $v \in \cO_r^{\oplus n}$ such that the set $\{v, x(v) , x^2(v), \ldots, x^{n-1} (v) \}$ is a basis of free $\cO_r$-module $\cO_r^{\oplus n}.$ The vector $v$ is called a cyclic vector of $x.$ 
The following lemma relates the cyclic and regular elements of $M_n(\mathbb F_q)$. 
\begin{lemma} 
\label{lem:regular-minimal-characteristic} 
An element $x \in M_n(\cO_1) = M_n(\mathbb F_q)$ is regular if and only $x$ is cyclic. 

\end{lemma} 
 \begin{proof}
This follows from the definition of regular $x \in M_n(\cO_1)$ (a proof appears in for example, Suprunenko-Tyshkevich~\cite[Theorem~5]{MR0201472}). 
\end{proof}
Next, we recall the definition of regular elements of $\g(\cO_r)$. 
\begin{definition} [Regular element of $\g(\cO_r)$] 
An element $x \in \g(\cO_r) \subseteq M_n(\cO_r)$ is called regular if and only if it is cyclic. 
\end{definition}

For any $i \leq r,$ there exists a natural projection $\rho_{r, i}: \cO_{r} \rightarrow \cO_i$ which is a ring homomorphism. By applying entry wise, we obtain a projection $\rho_{r,i}: \g(\cO_{r})  \rightarrow \g(\cO_i).$ For any $x \in \g(\cO_{r}),$ the image $\rho_{r,1}(x)$ is denoted by $\bar{x}.$ 
 
\begin{lemma}
\label{lem:hill-results-regular} Let $x \in \g(\cO_r) \subseteq M_n(\cO_r)$ and let $r \geq 1.$ The following are equivalent. 
\begin{enumerate}
\item The element $x \in \g(\cO_r)$ is regular. 
\item The projection of $x$ to $\g(\cO_i)$ for every $1 \leq i \leq r$ is regular.
\item The centralizer $C_{M_n(\cO_r)}(x) = \{ y \in M_n(\cO_r) \mid xy = yx \}$ of $x$ in $M_n(\cO_r)$ is abelian.
\item $C_{M_n(\cO_r)}(x) =  \cO_r[x],$ where $\cO_r[x]$ is a $\cO_r$-subalgebra of $M_n(\cO_r)$ generated by the set $\{ I, x, x^2, \ldots, x^{n-1} \}.$ 
\end{enumerate}

\end{lemma}
\begin{proof}
A proof of this follows from Hill~\cite[Theorem~3.6, Corollary~3.7]{MR1334228} along with the fact that $x \in \g(\cO_r)$ is regular if and only if $x$ is regular when viewed as an element of $M_n(\cO_r)$. 
\end{proof}

Recall $\cO_1 = \mathbb F_q.$ We use $d_{\g}$ to denote the dimension of the $\mathbb F_q$-vector space $\g(\cO_1).$ For $\bar{x} \in \mathrm{g}(\cO_1),$ the dimension of the centralizer algebra $C_{\g(\cO_1)}(\bar{x}) = \{ \bar{y} \in \g(\cO_1) \mid \bar{x}\bar{y} = \bar{y}\bar{x}  \}$ as $\cO_1$-vector space is denoted by $d_{\g(\cO_1)}(\bar{x}).$ Note that by Lemmas~\ref{lem:hill-results-regular} and \ref{lem:regular-minimal-characteristic}, for every regular $\bar{x} \in M_n(\cO_1)$ and $\bar{y} \in \sl_n(\cO_1)$ we have $d_{M_n(\cO_1)} (\bar{x}) = n$ and $d_{\sl_n(\cO_1)} (\bar{y}) = n-1.$ 
\begin{lemma}
Let $x \in \g(\cO_r)$ be a regular matrix and $C_{\g(\cO_r)}(x) = \{ y \in \g(\cO_r) \mid xy = yx  \}$ be the centralizer of $x$ in $\g(\cO_r).$ Then $|C_{\g(\cO_r)}(x)| = q^{ (d_{\g(\cO_1)}(\bar{x}) ) r}.$
\end{lemma} 
\begin{proof}  Both for $\g = M_{n}(\cO_{r})$ and $\g = \mathrm{sl}_n(\cO_r)$ and regular $x$, the result follows by Lemma~\ref{lem:hill-results-regular}.

\end{proof}

\begin{definition}($a$-regular element) For $a \in \cO_r^\times$, let $x \in M_n(\cO_r)$ be of the following form, 
	\begin{equation} \label{regular element}
	x = \left( 
	\begin{matrix}
	0 & 0 & 0 & \cdots & 0 & x_{1} \\
	a & 0 & 0 & \cdots & 0 & x_{2} \\
	0 & 1 & 0 & \cdots & 0 & x_{3} \\
	\vdots & \vdots & \vdots &  & \vdots & \vdots \\
	0 & 0 & 0 & \cdots & 1 & x_{n}
	\end{matrix}
	\right)
	\end{equation}
	where $x_{i} \in \cO_r$ for $1 \leq i \leq n.$ Then $x \in M_n(\cO_r)$ is regular and is called an $a$-regular element of $M_n(\cO_r)$.  
\end{definition}

\begin{lemma}
\label{lem:about-regular-elements} The following are true for the set of regular elements of $\g(\cO_r).$
\begin{enumerate}
\item The set of regular elements in $\g(\cO_r)$ is invariant under the conjugation action of $\G(\cO_r).$ 
\item Every conjugacy class of regular elements of  $\g(\cO_r)$ contains an $a$-regular element for some $a \in \cO_r^\times$.
\item For any $a \in \cO_r^\times$, the number of $a$-regular conjugacy classes of $\g(\cO_r)$ is  $q^{ (d_{\g(\cO_1)}(\bar{x}))  r},$ where $x \in \g(\cO_r)$ is any regular matrix. 
\end{enumerate} 

\end{lemma}

\begin{proof} The proof of $(1)$ follows by Lemma~\ref{lem:hill-results-regular}. For $\GL_n(\cO_r)$ every $a$-regular matrix is conjugate to a $1$-regular element via conjugating with a diagonal matrix $(a, 1, \ldots, 1)$. The rest of the proof of (2) for $\G = \GL_n$ and $\g = M_n$ follows from Mcdonald~\cite[p.417-419]{MR769104} (see also Hill~\cite[Remark~3.10]{MR1334228} ). For $\G = \SL_n$, $\g = \sl_n$ we first note that by above for $x_1 \in \sl_n(\cO_r),$ there exists $g \in \GL_n(\cO_r)$ such that $gx_1 g^{-1}$ is $1$-regular. Consider the diagonal matrix $z = (\frac{1}{\det(g)} , 1, \ldots, 1)\in \GL_n(\cO_r)$ and $g' = zg \in \SL_n(\cO_r)$ then $g' x_1 (g')^{-1} $ is $\det(g)$-regular. This in particular implies that every similarity class in $\sl_n(\cO_r)$ contains an $a$-regular element for some $a \in \cO_r^\times$. From this (2) and (3) for $\g = \sl_n$ follow easily from the corresponding results for $\g = M_n.$  
It is to be noted that a similarity class in $\g(\cO_r)$ may contain both an $a$-regular and a $b$-regular element for $a \neq b$. However for a fixed $a \in \cO_r^\times$, each similarity class contains at most one $a$-regular element. 

Last we note that the similarity class of an $a$-regular element is determined by its characteristic polynomial so $(3)$ follows by direct computations from Lemma~\ref{lem:hill-results-regular}. 
\end{proof}

\begin{lemma}
\label{lem:stab-unipotent-intersection}
Let $x \in M_n(\cO_r)$ be an $a$-regular matrix. Then $ C_{\g(\cO_r)} (x) \cap \mathbf{U}(\cO_r) = C_{\G(\cO_r)} (x) \cap \mathbf{U}(\cO_r) = \{ I_n \}$ 
\end{lemma}
\begin{proof}
 We note that for $a$-regular matrix $x,$ for every $1 \leq i \leq n-1,$ there exists $1 \leq j, k \leq n$ such that $j > k$ and $(j,k)^{th}$-entry of $x^i \in \cO_r^\times$. Now result follows by the fact that $ C_{\G(\cO_r)} (x) \subseteq C_{\g(\cO_r)} (x) \subseteq  C_{M_n(\cO_r)} (x)  ,$ by Lemma~\ref{lem:hill-results-regular} the set $C_{\g(\cO_r)} (x)$ is  generated by $\{ I, x, x^2, \ldots, x^{n-1}  \}$ as $\cO_r$-algebra and for every $y \in \mathbf{U}(\cO_r),$ the $(j,k)^{th}$ entry of $y$ is zero for $j > k.$ 
\end{proof}
\begin{lemma}
\label{lem:uppertriangular-regular}
Let $x = (x_{ij}) \in M_n(\cO_r)$ be such that $x_{21} = a \in \cO_r^\times $ and $x_{ij} = 1$ for all $i, j$ such that $i \geq 3$, $i = j + 1$ and $x_{ij} = 0$ for all $i, j$ such that $i-j \geq 2.$ Then $x$ is regular. 
\end{lemma}
\begin{proof} 
By Lemma~\ref{lem:hill-results-regular}, it is enough to prove that $\bar{x}$ is regular for all above $x.$ For that it is easy to see that $v = (1, 0, 0, \ldots, 0) \in \cO_1^{\oplus n}$ satisfies that $\{ v, \bar{x}v, \bar{x}^2v, \cdots, \bar{x}^{n-1}v \}$ generates the n-dimensional vectors space $\cO_1^{\oplus n}.$

\end{proof}

\section{Construction of regular representations of $\G(\cO_\ell)$}
\label{sec:regular-representation-construction}
In this section we follow \cite{MR3737836, MR3743488} to give a brief outline of the construction of regular representations of $\G(\cO_\ell)$, where either $\G = \SL_n$ with $(p,n) = (p,2) = 1$ or $\G = \GL_n$.

For $ i \leq \ell$ and the natural projection maps  $\rho_{\ell,i}: \G(\cO_{\ell})  \rightarrow \G(\cO_i),$ as defined in Section~\ref{sec:regular-elements}, let $\K_{\ell}^{i} =  \ker (\rho_{\ell,i}) $ be the $i$-th congruence subgroups of $\G(\cO_\ell).$ It is easy to note that for $i \geq \ell/2,$  the group  $\K^i_\ell $ is isomorphic to the abelian additive subgroup $\g(\cO_{\ell-i})$ of $M_n(\cO_{\ell-i}).$ Let $\varphi: \cO_\ell  \rightarrow \mathbb C^\times$ be a fixed primitive one dimensional representation of $\cO_\ell$. 
 For any $i \leq \ell/2$ and  $x \in \g(\cO_i),$ let $\hat{x} \in \g(\cO_\ell)$ be an arbitrary lift of $x$ satisfying $\rho_{\ell,i}(\hat{x}) = x.$ Define $\varphi_x: I + \varpi^{\ell-i} \g(\cO_\ell) \rightarrow \mathbb C^\times$ by 
 \[ 
 \varphi_x(I + \varpi^{\ell-i} y) = \varphi(\varpi^{\ell-i}\mathrm{tr}(\hat{x}y)),
 \]
for all $I + \varpi^{\ell-i} y \in \K^{\ell-i}_\ell.$  Then $\varphi_x$ is easily seen to be a well defined one dimensional representation of $\K^{\ell-i}_\ell .$ Further it is easy to see that for $i \geq \ell/2$ the following duality of abelian group $\K^{i}_\ell$ and $\g(\cO_{\ell-i})$ holds.
\begin{equation}
\label{eq:duality}
\g(\cO_{\ell-i}) \cong \widehat {\K^{i}_\ell}\,\,; x \mapsto \varphi_x\,\, \mathrm{where}, \,\, \varphi_x(y) = \varphi(\varpi^{\ell-i}\mathrm{tr}(\hat{x}y))  \,\, \forall \,\, y \in \K^{i}_\ell. 
\end{equation}

We say a one dimensional representation $\varphi_x \in \widehat{\K^{i} _\ell}$ for $i \geq \ell/2$ is regular if and only if $x \in M_n(\cO_{\ell-i})$ is a regular matrix. By Lemma~\ref{lem:hill-results-regular}, for $i \geq \ell/2$ the representation $\varphi_x \in \widehat{\K^{i} _\ell}$ is regular if and only if $\varphi_x|_{ \K^{\ell-1} _\ell}$ is regular. Recall an irreducible representation $\rho$ of $\G(\cO_\ell)$ is called regular if the orbit of its restriction to $\K_\ell^{\ell-1} $ consists of one dimensional representations $\varphi_x$ for regular $x$.

 Now we summarize the steps of construction of regular representations of $\G(\cO_\ell) $. 
For $\G(\cO_\ell) $ with odd $p$ we will outline the method as given in \cite{MR3737836} and we follow \cite{MR3743488} for $\GL_n(\cO_\ell)$ with $p=2.$ It might be useful to look at Figures~\ref{fig:ell.even} and \ref{fig:ell.odd} while going through the steps of construction. For more details on this see \cite{MR3737836, MR3743488}.  The construction in both cases is based on the tools of abstract Clifford theory and orbits. See \cite[Chapter~6]{MR0460423} for general results and their proofs regarding Clifford theory and \cite[Theorem~2.1]{Singla-Pooja} for the precise statements related to Clifford theory that are used in this article. The main point of the construction is that the case of $\G(\cO_\ell)$ where either $\G = \SL_n$ with $(p,n) = (p,2) = 1$ or $\G = \GL_n$ are good in the sense that for any regular one dimensional representation $\varphi_x$  of $\widehat{\K_\ell^{\lceil \ell/2\rceil}}$, the inertia group $I_{\G(\cO_\ell)}(\varphi_x) = \{ g \in \G(\cO_\ell) \mid \varphi_x^g \cong \varphi_x  \}  $ is "nice". In all these cases, the required inertia group although is  non-abelian but still is a product of an abelian group with a congruence subgroup and therefore all the representations of the required inertia group lying above $\varphi_x$ can be constructed. This combined with the Clifford theory gives the construction of all regular representations of $\G(\cO_\ell)$.

\subsection{Construction of regular representations of $\G(\cO_\ell)$ for $\ell = 2m$} Let $\varphi_x \in \widehat{\K_\ell^m}$ be a regular one dimensional representation of $\K_{\ell}^{m}$ for $x \in \g(\cO_m).$ Then the following are true. 

\medskip
\noindent {\bf E.1} Let   $I_{\G(\cO_\ell)}(\varphi_x) = \{ g \in \G(\cO_\ell) \mid \varphi_x^g \cong \varphi_x  \}  $ be the inertia group of $\varphi_x$ in $\G(\cO_\ell).$ 
\begin{enumerate} 
\item Then $I_{\G(\cO_\ell)}(\varphi_x)  =  \tC_{\G(\lri_\ell)}(\tilde{x})       \K_\ell^{m} ,$ where $\tilde{x} \in \cO_\ell$ is any lift of $x$ to $\g(\cO_\ell).$  

\item Although the inertia group $I_{\G(\cO_\ell)}(\varphi_x)$ is not abelian but the quotient  $I_{\G(\cO_\ell)}(\varphi_x)/\K_\ell^m \cong \tC_{\G(\lri_m)}(x)$ is abelian by Lemma~\ref{lem:hill-results-regular}.   
\end{enumerate}
\medskip
\noindent {\bf E.2} Let $\delta$ be a one dimensional representation of $\tC_{\G(\lri_\ell)}(\tilde{x})$ extending $\varphi_x|_{\K^{m}_\ell \cap \tC_{\G(\lri_\ell)}(\tilde{x})}$. The existence of this $\delta$ follows easily because the group $\tC_{\G(\lri_\ell)}(\tilde{x})  $ is abelian. Define $\widetilde{\varphi_x}: I_{\G(\cO_\ell)}(\varphi_x)  \rightarrow \mathbb {C}^\times$ by $\widetilde{\varphi_x}(ab) = \delta(a) \varphi_x(b)$ for all $a \in \tC_{\G(\lri_\ell)}(\tilde{x}) $ and $b \in  \K_\ell^{m} $. It can be easily shown that $\widetilde{\varphi_x}$ is a well defined one dimensional representation of $I_{\G(\cO_\ell)}(\varphi_x)$ that extends the representation $\varphi_x$. 

\medskip
\noindent {\bf E.3} Let $\rho \in \mathrm{Irr} \left( \G(\cO_\ell)  \mid \varphi_x     \right)$ be a regular representation of $\G(\cO_\ell),$ then there exists an extension $\widetilde{\varphi_x}$ of $\varphi_x$ to $I_{\G(\cO_\ell)}(\varphi_x)$ such that $\rho \cong \mathrm{Ind}_{I_{\G(\cO_\ell)}(\varphi_x)}^{\G(\cO_\ell)} (\widetilde{\varphi_x}) .$ This follows by ${\bf E.1}$, ${\bf E.2}$ combined with Clifford theory. 

\medskip
\noindent {\bf E.4}   We have $\left| \mathrm{Irr} (\G(\cO_\ell)  \mid \varphi_x) \right| = \left|\tC_{\G(\lri_m)}(x) \right|.$ 

\medskip
\noindent {\bf E.5}
Every $\rho \in \mathrm{Irr} \left(  \G(\cO_\ell)  \mid \varphi_x     \right)$ has dimension $\dfrac{|\G(\cO_\ell)|}{|\tC_{\G(\lri_m)}(x)||\K_\ell^m|}.$

\begin{figure}[htb!]
\centering
\begin{minipage}{0.45\textwidth}
        \centering
         
         \label{fig:ell.even}
        \begin{displaymath}
        \xymatrix{
        &  {\G(\lri_\ell)}\ar@{-}[d]  &                        \\
           &     {I_{\G(\lri_\ell)}(\varphi_x)}\ar@{-}[dr] \ar@{-}[dl]   &            \\ 
        {\K^m_\ell}\ar@{-}[dr]  &     &  \tC_{\G(\lri_\ell)}(\tilde{x})\ar@{-}[dl]      \\
        & {\K^{m}_\ell \cap \tC_{\G(\lri_\ell)}(\tilde{x})}\ar@{-}[d] &   \\
         & {\{1\}}& }
        \end{displaymath}
        \caption{Even case $\ell=2m$}
        \end{minipage}
 \begin{minipage}{.5\textwidth}
        \centering

 \label{fig:ell.odd}
\begin{displaymath}
\xymatrix{
&  {\G(\lri_\ell)}\ar@{-}[d]  &                                &                    \\
   &     {I_{\G(\lri_\ell)}(\sigma)}\ar@{-}[dd] \ar@{-}[dl]   &                              &    \\ 
{H^m_\ell}\ar@{-}[d]^{q^{\frac{d_{\g}-d_{\g(\cO_1)}(\xx) }{2}}}   &     &  &      \\
 {J_x}\ar@{-}[d]^{q^{\frac{d_{\g}-d_{\g(\cO_1)}(\xx) }{2}}}  &                                  \tC_{\G(\lri_\ell)}(\tilde{x})   \ar@{-}[dl] &    & \\
R_{\xx}  \ar@{-}[d]   &                              &     & \\
 {\K^{m+1}_\ell}\ar@{-}[dr]  &   &         &\\
& {\K^{m+1}_\ell \cap \tC_{\G(\lri_\ell)}(\tilde{x})}\ar@{-}[d] \ar@{-}[uuu] &  & \\
 & {\{1\}}& & 
}
\end{displaymath}
 \caption{Odd case $\ell=2m+1$}
\end{minipage}%
\end{figure}

\subsection{Construction of regular representations of $\G(\cO_\ell)$ for $\ell = 2m + 1$} The construction for this case is involved as compared to $\ell = 2m$ case. So here we highlight what is important for us and refer the reader to \cite{MR3737836, MR3743488} for more details. Let $\varphi_x \in \widehat{\K_\ell^{m+1}}$ be a regular one dimensional representation of $\K_{\ell}^{m+1}$ for $x \in \g(\cO_{\ell-m}).$ Let $\tilde{x} \in \g(\cO_\ell)$ be a lift of $x$ to $\g(\cO_\ell).$ For $p \neq 2,$ either $\G = \SL_n$ with $p \nmid n$ or $\G = \GL_n.$ For $p=2,$ we consider $\G = \GL_n.$

\medskip
\noindent { \bf O.1} Define groups $H^m_\ell$ and $R_{\xx}$ as follows.
\[
H^i_\ell = \begin{cases}\K_\ell^i, & p \neq 2 \\ 
 \left(\K^{1}_\ell \cap \tC_{\G(\lri_\ell)}(\tilde{x})\right) \K_\ell^{i}, & p =2.
  \end{cases}
\]

 \[
 R_{\xx} = \begin{cases} \left(\K^{m}_\ell \cap \tC_{\G(\lri_\ell)}(\tilde{x})\right) \K_\ell^{m+1}, & p \neq 2 \\ 
 \left(\K^{1}_\ell \cap \tC_{\G(\lri_\ell)}(\tilde{x})\right) \K_\ell^{m+1}, & p =2.
 
 \end{cases}
 \]
 \medskip
 \noindent { \bf O.2}
 The following are true for groups $H^i_\ell$ and $R_{\xx}.$

\begin{enumerate}
\item The group $R_{\xx}$ is a normal subgroup of $H_{\ell}^{m}.$ 
\item The quotient group $R_{\xx}/\K_\ell^{m+1} $ is abelian. 
\item The one dimensional representation $\varphi_x \in \widehat{K^{m+1}_\ell }$ extends to $R_{\xx}.$
\end{enumerate}

\medskip
\noindent { \bf O.3} The extensions of $\varphi_x$ to $R_{\xx}$ satisfy the following: 
\begin{enumerate} 
\item each extension of $\varphi_x$ to $R_{\bar{x}}$ is stable under $\K^{m}_\ell.$

\item Each extension of $\varphi_x$  to $R_{\bar{x}}$ determines a non-degenerate bilinear form on $H^{\ell}_m/R_{\bar{x}}$. Then there exists a ``nice'' maximal isotropic space $J_x$ with respect to this bilinear form such that, parallel to the method of  construction of ``Heisenberg" groups, every extension of $\varphi_x$ to $R_{\bar{x}}$ determines a unique irreducible representation of $H^m_\ell$ of dimension $q^{\frac{d_{\g} - d_{\g(\cO_1) }(\xx) }{2}},$ see \cite[Section~3]{MR3737836} for $p \neq 2$ and \cite[Section~3]{MR3743488} for $p=2$ for more details.

\item Let $\widetilde{\varphi_x}$ be an extension of $\varphi_x$ to $R_{\xx}$ and $\sigma \in \mathrm{Irr}(H_\ell^m \mid \varphi_x) $ be unique irreducible representation determined by $\widetilde{\varphi_x}.$ Then, 
\[\sigma|_{R_{\xx}} = \underbrace{\widetilde{\varphi_x} + \cdots + \widetilde{\varphi_x} }_{\frac{d_{\g}-d_{\g(\cO_1)}(\xx) }{2}-\mathrm{times}}. 
\]  
\end{enumerate} 

\medskip
\noindent { \bf O.4} Let $I_{\G(\cO_\ell)}(\sigma) = \{ g \in \G(\cO_\ell) \mid \sigma^g \cong \sigma  \}$ be the inertia groups of $\sigma \in \mathrm{Irr}(H_\ell^m \mid \varphi_x).$ Then the following are true:
\begin{enumerate} 
\item $I_{\G(\cO_\ell)}(\sigma)  = I_{\G(\cO_\ell)}(\varphi_x) = \tC_{\G(\lri_\ell)}(\tilde{x}) \K_\ell^m,$  where $\tilde{x} \in \g(\cO_\ell)$ is any lift of $x$ to $\g(\cO_\ell).$ This follows by the definition of $R_{\bar{x}}$ along with {\bf O.2}. 
\item Every $\sigma \in \mathrm{Irr}(H_\ell^m \mid \varphi_x)$ extends to the inertia group $I_{\G(\cO_\ell)}(\sigma) .$ In particular, every such extension induces irreducibly to $\G(\cO_\ell)$ and gives rise to a regular representation of $\G(\cO_\ell).$ This is slightly technical part and depends on the structure of groups $R_{\bar{x}}$, $J_x$, $H_\ell^m$ and $I_{\G(\cO_\ell)}(\sigma)$. 
\end{enumerate}

\medskip
\noindent { \bf O.5} Combining above all, there exists a bijection  $
\Irr\left( \G(\cO_\ell) \mid \sigma \right)   \leftrightarrow  \mathrm{C}_{\G(\cO_{m})}(x).$ 

\medskip
\noindent { \bf O.6} 
Every $\rho \in \mathrm{Irr} \left( \G(\cO_\ell) \mid \sigma \right)$ for  $\sigma \in \mathrm{Irr}(\K_\ell^m \mid \varphi_x)$ has dimension $q^{\frac{d_{\g}-d_{\g(\cO_1)}(\xx) }{2}}\frac{|\G(\cO_\ell)|}{|\mathrm{C}_{\G(\cO_{m})}(x)| |\K_\ell^m|} .$ 

\medskip
\noindent { \bf O.7} Altogether, we have the following:
\begin{enumerate}
\item $|\mathrm{Irr} \left( \G(\cO_\ell) \mid \varphi_x \right)|  = q^{d_{\g(\cO_1)}(\xx)} |\mathrm{C}_{\G(\cO_{m})}(x)| .$
\item Every $\rho \in \mathrm{Irr} \left( \G(\cO_\ell) \mid \varphi_x \right)$ has dimension $q^{\frac{d_{\g}-d_{\g(\cO_1)}(\xx) }{2}}\frac{|\G(\cO_\ell)|}{|\mathrm{C}_{\G(\cO_{m})}(x)| |\K_\ell^m|} .$ 
\end{enumerate}

\section{Uniqueness of $\theta$-Whittaker models}
In this section first we prove Theorem~\ref{main-theorem-2} and then use this to prove Theorem~\ref{thm:main-theorem}.  Let $\varphi$ be a fixed primitive character of $\cO_\ell.$ For $a \in \cO_{\ell}^\times$, define $\varphi_{a} : \cO_\ell \rightarrow \C^{\times}$ by $\varphi_a(x) = \varphi(ax)$. Then it is easy to see that $\varphi_a$ is a primitive character of $\cO_\ell$ and every primitive character of $\cO_\ell$ is of the form $\varphi_a$ for some $a \in \cO_\ell^\times$. 
For $a \in \cO_{\ell}^\times$, define $\theta_{a} : \mathbf{U}(\cO_\ell) \rightarrow \C^{\times}$ by
\begin{equation}
\theta_{a}((x_{ij})) := \varphi (a x_{12} + x_{23} + \cdots + x_{(n-1)n}).
\end{equation}
The subgroup of $\G(\Ol)$ consisting of all diagonal matrices (split torus) acts on $\mathbf{U}(\cO_\ell)$ by conjugation. It is to be noted that every one dimensional non-degenerate representation of $\mathbf{U}(\cO_\ell)$ is equivalent to $\theta_a$ for some $a \in \cO_\ell^\times$ under the conjugation action of the split torus. Therefore without loss of generality, it is enough to prove Theorems~\ref{thm:main-theorem} and \ref{main-theorem-2}  for $\theta = \theta_a$ for all $a \in \cO_\ell^\times$ and for any fixed primitive character $\varphi$ of $\cO_\ell$. 

\subsection{Proof of Theorem~\ref{main-theorem-2}}
\label{sec:proof-of-whittaker-model} 
Throughout this section we assume that either $\G = \SL_n$ with $(p,n) = (p,2) = 1$ or $\G = \GL_n$. For $1 \leq k < \ell,$ let $\mathbf{U}(\varpi^{k} \Ol)$ be the group of unipotent upper triangular matrices for which the entries in the strictly upper triangular part belong to $\varpi^{k}\Ol.$ We assume that $m$ is such that $\ell = 2m$ for even $\ell$ and $\ell = 2m+1$ for odd $\ell$. We also assume that primitive character $\varphi$ of $\cO_\ell$ used in the definition of the construction of regular representations of $\G(\cO_\ell)$ in Section~\ref{sec:regular-representation-construction} and in the definition of $\theta_a$ is the same one. The following definition plays an important role in what follows.

\begin{definition}\label{def:regular-representation}($\alpha$-regular representation of $\G(\cO_\ell)$) A regular representation $\rho$ of $\mathrm{G}(\cO_\ell)$ is called $\alpha$-regular for $\alpha \in \cO_{\lfloor \frac{\ell}{2} \rfloor}$ if there exists an $\alpha$-regular element $x \in \g(\cO_{\lfloor \frac{\ell}{2} \rfloor})$ such that $(\rho|_{K_\ell^{\lceil \frac{\ell}{2}\rceil}}, \varphi_x) \neq 0$ for $\varphi_x \in \widehat{K_\ell^{\lceil \frac{\ell}{2}\rceil}}$.
\end{definition}

We proceed to state and prove the main theorem of this article. As mentioned above, every non-degenerate character of $\mathbf U(\cO_\ell)$ is conjugate to $\theta_a$ for some $a \in \cO_\ell^\times$ and therefore Theorem~\ref{main-theorem-2} directly follows from Theorem~\ref{thm: existence}. For $\G(\cO_\ell) = \GL_n(\cO_\ell)$ and for any $a \in \cO_\ell^\times$, every regular representation is $\rho_{\ell, m}(a)$ regular. Therefore Corollary~\ref{cor:hill-regular-iff-whittaker} also follows from Theorem~\ref{thm: existence}.  
\begin{theorem}  \label{thm: existence}
\begin{enumerate}
\item An irreducible representation $\pi$ of $\G(\Ol)$ for $\ell \geq 2$ admits a $\theta_a$-Whittaker model if and only if it is $\rho_{\ell, m }(a)$-regular.
\item An irreducible representation of $\G(\cO_\ell)$ for $\ell \geq 2$ admits at most one $\theta_a$-Whittaker model for every $a \in \cO_\ell^\times$. 
\end{enumerate}
\end{theorem}

\begin{proof}

First, we prove that for $a \in \cO_\ell^\times$, every $\rho_{\ell, m}(a)$-regular representation admits a $\theta_a$-Whittaker model. Define $N_\ell^m$ as follows. 
\[
N_\ell^m = \begin{cases} K_\ell^m, & \ell = 2m \\ 
 H_\ell^m, & \ell = 2m+1,
  \end{cases}
\]
where $K_\ell^m$ and $H_\ell^m$ as defined in the Section~\ref{sec:regular-representation-construction}. 
Since $\pi$ is a $\rho_{\ell, m }(a)$-regular representation of $\G(\OO_{\ell})$ by construction of $\pi$ (see {\bf E.3} and { \bf O.4 }) there exists an irreducible representation $\sigma$ of $N_{\ell}^{m}$ which lies over a regular one dimensional representation $\varphi_{x}$ of $\sfK_{\ell}^{\lceil \frac{\ell}{2}\rceil}$ for some $\rho_{\ell, m}(a)$-regular element $x$ and $\sigma$ extends to the inertia group $I_{\G(\OO_{\ell})}(\sigma)$ written as $\widetilde{\sigma}$ with the property that 
\[
\pi \cong \Ind_{I_{\G(\OO_{\ell})}(\sigma)}^{\G(\OO_{\ell})} \widetilde{\sigma}.
\]

Write $\so$ for the inertia group $I_{\G(\OO_{\ell})}(\sigma).$  By using Frobenius reciprocity and Mackey theory we get
\[
\Hom_{\G(\OO_{\ell})} \left(\pi, \Ind_{\mathbf{U}(\Ol)}^{\G(\Ol)}(\theta_a) \right)  =  \bigoplus_{g \in \so \backslash \G(\OO_{\ell}) / \mathbf{U}(\OO_{\ell})}  \Hom_{\mathbf{U}(\OO_{\ell}) \cap (\so)^{g}} \left(\widetilde{\sigma}^{g}, \theta_a \right)
\]
To prove $\Hom_{\G(\OO_{\ell})} \left(\pi, \Ind_{\mathbf{U}(\Ol)}^{\G(\Ol)}(\theta_a) \right) \neq 0$ it is enough to prove that $\Hom_{\mathbf{U}(\OO_{\ell}) \cap \so} \left(\widetilde{\sigma}, \theta_a \right) \neq 0$ and this we prove in next few lemmas.

\begin{lemma}
\label{lem:intersection}
For $\ell = 2m$, let $x \in \g(\cO_{m})$ be a $\rho_{\ell, m}(a)$-regular element. The following are true. 
\begin{enumerate}
\item The intersection $\mathbf{U}(\OO_{\ell}) \cap I_{\G(\OO_{\ell})} ( \sigma) = \mathbf{U}(\varpi^{m} \OO_{\ell}).$
\item We have $\Hom_{ \mathbf{U}(\varpi^{m} \OO_{\ell})} \left( \widetilde{\sigma}, \theta_{a } \right) \neq 0.$
\end{enumerate}
\end{lemma}

\begin{proof}
For $\ell = 2m$, we have $N_\ell^m = \K_\ell^m$ and therefore $\sigma = \varphi_x$ in this case. Then $I_{\G(\cO_\ell)}(\varphi_x) = \tC_{\G(\lri_\ell)}(\tilde{x}) \K_\ell^m$ for any lift $\tilde{x}$ of $x$ by {\bf E.1}. it is clear that $\mathbf{U}(\varpi^{m} \OO_{\ell}) \subseteq \mathbf{U}(\OO_{\ell}) \cap I_{\G(\OO_{\ell})} ( \varphi_{x}) ,$ so we proceed to prove the other side inclusion. 
The element $x$ is $\rho_{\ell, m }(a)$-regular. Therefore by Lemma~\ref{lem:stab-unipotent-intersection}, we have $\mathbf{U}(\cO_\ell) \cap C_{\G(\cO_\ell)} (\tilde{x}) = \{ 1 \} .$ This implies,  any $v \in \mathbf{U}(\OO_{\ell}) \cap I_{\G(\OO_{\ell})} ( \varphi_{x})$ satisfies $v \in \K_\ell^m$ This along with $v \in \mathbf{U}(\cO_\ell)$ gives $v \in \mathbf{U}(\varpi^{m} \OO_{\ell}).$ Therefore (1) follows.

The proof of (2) amounts to say that the restriction of $\widetilde{\varphi}_{x}$ to $\mathbf{U}(\varpi^{m} \OO_{\ell})$ is same as the restriction of $\theta_a$ to $\mathbf{U}(\varpi^{m} \OO_{\ell}).$
Which is clear since $x$ is $\rho_{\ell, m}$-regular and therefore $\widetilde{\varphi_x}(y) = \theta_a(y)$ for every $y \in \mathbf{U}(\varpi^{m} \OO_{\ell}).$
This proves part (2).
\end{proof}

\begin{lemma}
\label{lem:decomposition-of-sigma}
For $\ell = 2m +1 ,$ consider the set $\r = \mathrm{Irr}\left( \mathbf{U}(\varpi^m \cO_\ell ) \mid \theta_a|_{\mathbf{U}(\varpi^{m+1} \cO_\ell ) }  \right).$ The following are true. 
\begin{enumerate} 
\item The set $\r$ consists of one dimensional representations of $\mathbf{U}(\varpi^m \cO_\ell ).$ 
\item If  $x \in \g(\cO_m)$ is a $\rho_{\ell, m}(a)$-regular element and $\sigma \in \mathrm{Irr} \left( H_\ell^m \mid \varphi_x \right).$ Then, 
\[
\sigma|_{\mathbf{U}(\varpi^m \cO_\ell )} \cong \underset{\chi \,\, \in \,\, \r}{\oplus} \,\,\chi.  
\]
\end{enumerate} 
\end{lemma}  

\begin{proof} The one dimensional representation $\theta_a|_{\mathbf{U}(\varpi^{m} \cO_\ell )}$ extends the character $\theta_a|_{\mathbf{U}(\varpi^{m+1} \cO_\ell )}$ to group $\mathbf{U}(\varpi^m \cO_\ell ).$ The quotient group $\frac{\mathbf{U}(\varpi^m \cO_\ell )}{\mathbf{U}(\varpi^{m+1} \cO_\ell ) } $ is abelian, so (1) follows from Clifford theory. For (2), consider the normal subgroup  $R_{\xx} = {\K^{m}_\ell \cap \tC_{\G(\lri_\ell)}(\widetilde{x})}\K_\ell^{m+1}$ of $\K_\ell^m.$ By ${\bf O. 3},$ for $\sigma \in \mathrm{Irr} \left( \K_\ell^m \mid \varphi_x \right),$ there exists an extension $\widetilde{\varphi_x}$ of $\varphi_x$ to $R_{\xx}$ such that $\sigma|_{R_{\xx}} = \underbrace{\widetilde{\varphi_x} + \cdots + \widetilde{\varphi_x} }_{(d_{\g}-d_{\g(\cO_1)}(\bar{x}))/2-\mathrm{times}}$ and $\sigma$ is unique irreducible representation of $\K_\ell^m$ lying above $\widetilde{\varphi_x}.$ 

Consider the hierarchy of groups given in Figure~\ref{figure:decomposition-to-U}. By Lemma~\ref{lem:stab-unipotent-intersection}, we have ${R_\xx} \cap \mathbf{U}(\varpi^{m}\cO_\ell ) = \mathbf{U}(\varpi^{m+1}\cO_\ell ).$ Since $x$ is $\rho_{\ell, m}(a)$-regular, we have 
\[
\widetilde{\varphi_x}|_{\mathbf{U}(\varpi^{m+1}\cO_\ell )} = {\varphi_x}|_{\mathbf{U}(\varpi^{m+1}\cO_\ell )} = \theta_a|_{\mathbf{U}(\varpi^{m+1}\cO_\ell)} 
\]
and $\widetilde{\varphi_x}$ is stable under $\mathbf{U}(\varpi^{m}\cO_\ell ),$ so the map  $\widetilde{\varphi_x} \circ\theta_a : (R_\xx)(\mathbf{U}(\varpi^{m}\cO_\ell )) \rightarrow \mathbb C^\times$ defined by $(\widetilde{\varphi_x} \circ \theta_a)(xy) = \widetilde{\varphi_x}(x) \theta_a(y)$ for all $x \in R_\xx$ and $y \in \mathbf{U}(\varpi^{m}\cO_\ell )$ is a well defined one dimensional representation of $(R_\xx)(\mathbf{U}(\varpi^{m}\cO_\ell )) .$ Further, the one dimensional representation $\widetilde{\varphi_x} \circ \theta_a$ extends $\widetilde{\varphi_x}.$ The quotient group $(R_\xx)(\mathbf{U}(\varpi^{m}\cO_\ell ))/R_\xx$ is easily see to be abelian of order $q^{\frac{d_{\g}-d_{\g(\cO_1)}(\bar{x})}{2}}.$ Therefore by Clifford theory, the set $\mathrm{Irr} \left( (R_\xx)(\mathbf{U}(\varpi^{m}\cO_\ell ))   \mid \widetilde{\varphi_x} \right) $ consists of exactly $q^{\frac{d_{\g}-d_{\g(\cO_1)}(\bar{x})}{2}}$-many distinct one dimensional representations. 

We claim that any $\chi \in \mathrm{Irr} \left( (R_\xx)(\mathbf{U}(\varpi^{m}\cO_\ell ))   \mid \widetilde{\varphi_x} \right) $ satisfies $\mathrm{Ind}_{(R_\xx)(\mathbf{U}(\varpi^{m}\cO_\ell ))} ^{H_\ell^m} (\chi) \cong \sigma.$ Recall $\sigma$ is a unique irreducible representation of $H_\ell^m$ lying above $\widetilde{\varphi_x}$ and therefore we must have  $\langle \sigma, \mathrm{Ind}_{(R_\xx)(\mathbf{U}(\varpi^{m}\cO_\ell )} ^{H_\ell^m} (\chi) \rangle \neq 0 $ for every $\chi \in  \mathrm{Irr} \left( (R_\xx)(\mathbf{U}(\varpi^{m}\cO_\ell ))   \mid \widetilde{\varphi_x} \right) $ . Now the claim follows because 
\[\dim(\sigma) = \dim(\mathrm{Ind}_{(R_\xx)(\mathbf{U}(\varpi^{m}\cO_\ell ))} ^{H_\ell^m} (\chi)) = {q^{\frac{d_{\g}-d_{\g(\cO_1)}(\bar{x})}{2}}}, 
\]
for every $\chi \in  \mathrm{Irr} \left( (R_\xx)(\mathbf{U}(\varpi^{m}\cO_\ell ))   \mid \widetilde{\varphi_x} \right) .$ 

\begin{figure}[htb!]
\label{figure:decomposition-to-U}
\centering
\begin{displaymath}
\xymatrix{
&  H_\ell^m\ar@{-}[d]^{q^{\frac{d_{\g}-d_{\g(\cO_1)}(\bar{x})}{2}}}  &                        \\
   &    (R_\xx)(\mathbf{U}(\varpi^{m}\cO_\ell )) \ar@{-}[dr]^{q^{d_{\g(\cO_1)}(\bar{x})}} \ar@{-}[dl]^{q^{\frac{d_{\g}-d_{\g(\cO_1)}(\bar{x})}{2}}}   &            \\ 
{R_\xx}\ar@{-}[dr]  &     &  \mathbf{U}(\varpi^{m}\cO_\ell )\ar@{-}[dl]      \\
& \mathbf{U}(\varpi^{m+1}\cO_\ell )\ar@{-}[d] &   \\
 & {\{1\}}& }
\end{displaymath}
\caption{}
 \end{figure}

\end{proof}

\begin{lemma}
For $\ell = 2m+1,$ let $x \in \g(\cO_m)$ be a $\rho_{\ell, m }(a)$-regular element and $\sigma \in \mathrm{Irr}(H_{\ell}^{m} \mid \varphi_{x}).$ 
Then the following are true. 
\begin{enumerate}
\item The intersection $\mathbf{U}(\OO_{\ell}) \cap I(\sigma) =  \mathbf{U}(\varpi^{m} \cO_\ell).$
\item We have, $\Hom_{\mathbf{U}(\varpi^{m} \OO_{\ell})} \left(\widetilde{\sigma}, \theta_a \right) = \Hom_{\mathbf{U}(\varpi^{m} \OO_{\ell})} \left(\sigma, \theta_a \right) \neq 0.$
\end{enumerate} 
\end{lemma}

\begin{proof}
For (1), we note that  by ${\bf O.4}$, $I(\sigma) = I_{\G(\OO_{\ell})} (\varphi_{x}) .$ The rest of the proof is similar to Lemma~\ref{lem:intersection}(1). For (2), note that the dimension of $\sigma$ is $q^{\frac{d_{\g}-d_{\g(\cO_1)}(\bar{x})}{2}}.$
By Lemma~\ref{lem:decomposition-of-sigma}, we have $\sigma|_{\mathbf{U}(\varpi^{m} \OO_{\ell})}$ is direct sum of $ {q^{\frac{d_{\g}-d_{\g(\cO_1)}(\bar{x})}{2}}}$ distinct one dimensional representations of $\mathbf{U}(\varpi^{m} \OO_{\ell}).$
We also know that $\sigma|_{\sfK_{\ell}^{m +1}}$ is isomorphic to direct sum of $q^{\frac{d_{\g}-d_{\g(\cO_1)}(\bar{x})}{2}}$ times $\varphi_{x}.$ 
Therefore all the one dimensional representations of $\sigma|_{\mathbf{U}(\varpi^{m} \OO_{\ell})}$ are extensions of $\varphi_{x}|_{\mathbf{U}(\varpi^{m+1} \OO_{\ell})}.$
Note that the index $[\mathbf{U}(\varpi^{m} \OO_{\ell}) : \mathbf{U}(\varpi^{m +1} \OO_{\ell})] = q^{\frac{d_{\g}-d_{\g(\cO_1)}(\bar{x})}{2}} = q^{\frac{n(n-1)}{2}}$ which is the same as the dimension of $\sigma.$ 
Thus $\sigma|_{\mathbf{U}(\varpi^{m} \OO_{\ell})}$ is direct sum of all the possible extensions of $\varphi_{x}|_{\mathbf{U}(\varpi^{m +1} \OO_{\ell})}$ to the subgroup $\mathbf{U}(\varpi^{m} \OO_{\ell}).$ 
By hypothesis, the element $x$ is $\rho_{\ell, m }(a)$-regular therefore $\theta_a|_{\mathbf{U}(\varpi^{m +1} \OO_{\ell})} = \varphi_{x}|_{\mathbf{U}(\varpi^{m +1} \OO_{\ell})}$. This implies that $\theta_a|_{\mathbf{U}(\varpi^{m} \OO_{\ell})}$ is an extension of $\varphi_{x}$ to $\mathbf{U}(\varpi^{m} \OO_{\ell})$. This combined with the above discussion regarding $\sigma|_{\mathbf{U}(\varpi^{m} \OO_{\ell})}$ gives 
$\Hom_{\mathbf{U}(\varpi^{m} \OO_{\ell})} (\sigma, \theta_a) \neq 0.$
\end{proof}

This completes the proof of the fact that for $a \in \cO_\ell^\times$ every $\rho_{\ell, m}(a)$ regular representation of $\G(\cO_\ell)$ admits a $\theta_a$-Whittaker model. Next, we prove that the sum of the dimensions of all inequivalent $\rho_{\ell, m }(a)$-regular representations of $\G(\Ol)$ is equal to the dimension of the induced representation $\Ind_{\mathbf{U}(\Ol)}^{\G(\Ol)}(\theta_a).$ 

\begin{lemma}
Let $\mathcal{R}_a$ be the set of isomorphism classes of $\rho_{\ell, m }(a)$-regular  representations of $\G(\OO_{\ell})$ and $\bar{x} \in \g(\mathbb F_q)$ be a regular element. Then the following is true. 
\[
\sum_{\pi \in \mathcal{R}_a} \dim(\pi) = \begin{cases} q^{\left(d_{\g(\cO_1)}(\bar{x}) \right)m} | \G(\cO_m) |, & \mathrm{for} \,\, \ell = 2m \\  
q^{ \left(d_{\g(\cO_1)}(\bar{x})\right)m} \cdot q^{\frac{ d_{\g}+\left(d_{\g(\cO_1)}(\bar{x})\right)}{2}} | \G(\cO_m) |, & \mathrm{for} \,\, \ell = 2m+1. 
\end{cases}
\]
\end{lemma}
\begin{proof}

We prove this result for cases $\ell = 2m$ and $\ell = 2m+ 1$ separately. 

For $\ell = 2m,$ by ${\bf E.4}$ and ${\bf E.5},$ the sum of dimensions of all inequivalent $\rho_{\ell, m }(a)$-regular representations lying above the one dimensional representation $\varphi_{x}$ for any $\rho_{\ell, m }(a)$-regular $x \in \g(\cO_m)$ is the index $[\G(\OO_{\ell}) : \sfK_{\ell}^{m}]= |\G(\cO_m) |.$
We recall that, the number of $\rho_{\ell, m }(a)$-conjugacy classes of regular elements in $\g(\cO_m)$ is $q^{ \left(d_{\g(\cO_1)}(\bar{x}) \right) m}$ (by Lemma \ref{lem:about-regular-elements}). Therefore
$\sum_{\pi \in \mathcal{R}_a} \dim(\pi) $ is $q^{ \left(d_{\g(\cO_1)}(\bar{x}) \right) m} |\G(\cO_m) |$ for $\ell = 2m.$ 

For $\ell = 2m+ 1,$ by ${\bf O.7},$ the sum of the dimensions of all inequivalent $\rho_{\ell, m }(a)$-regular representations lying over the one dimensional representation $\varphi_{x}$ of $\K_\ell^{m+1}$ for any $\rho_{\ell, m}(a)$-regular $x \in \g(\cO_m)$ is $q^{\frac{ d_{\g}+\left(d_{\g(\cO_1)}(\bar{x})\right)}{2}}|\G(\cO_m) |.$ Multiplying this with the number of $\rho_{\ell, m }(a)$-regular conjugacy classes of $\g(\cO_m),$ we obtain our result even in this case. 
\end{proof} 
Next, we proceed to prove that the dimension of $\Ind_{\mathbf{U}(\Ol)}^{\G(\Ol)}( \theta_a)$ is same as $\sum_{\pi \in \mathcal{R}_a} \dim(\pi)$ for all $\ell.$ 
We note that $| \mathbf{U}(\OO_{\ell}) | = q^{\frac{n(n-1)}{2}\ell} = q^{\frac{d_{\g} - \left(d_{\g(\cO_1)}(\bar{x})\right)}{2} \ell} $ for all $\ell.$ Consider the short exact sequence,  \[
1 \rightarrow \sfK_{\ell}^{m} \rightarrow \G(\OO_{\ell}) \rightarrow \G(\cO_m) \rightarrow 1.
\]
Therefore, $| \G(\OO_{\ell}) | = |\K^m_\ell| | \G(\cO_m) | $ and we have the following, 
\[
|\K^m_\ell| = \begin{cases}   q^{(d_{\g}) m}, & \mathrm{for}\,\, \ell = 2m\\   q^{d_{\g}(m +1)}, & \mathrm{for} \,\, \ell = 2m+ 1.\end{cases}
\]
Since, the dimension of $\Ind_{\mathbf{U}(\Ol)}^{\G(\Ol)}(\theta)$ is equal to $[\G(\OO_{\ell}) : \mathbf{U}(\OO_{\ell})],$ we obtain the result by substituting $|\G(\OO_{\ell})|$ and $|\mathbf{U}(\OO_{\ell})|$ from above. We have proved that every $\pi \in \mathcal{R}_a$ is a constituent of  $\Ind_{\mathbf{U}(\Ol)}^{\G(\Ol)}(\theta)$ and the dimension of $\Ind_{\mathbf{U}(\Ol)}^{\G(\Ol)}( \theta_a)$ is equal to $\sum_{\pi \in \mathcal{R}_a} \dim(\pi)$ for all $\ell \geq 2.$ Both of these facts together complete the proof of Theorem~\ref{thm: existence}.  
\end{proof} 

\subsection{Proof of Theorem~\ref{thm:main-theorem}}
\label{sec:proof-of-multiplicity-one}

Now we are in a position to complete the proof of Theorem~\ref{thm:main-theorem}.  For $\ell \geq 2$ and for the groups $\G = \GL_n$ or $\G = \SL_n$ with $(p, n) = (p,2) = 1,$ Theorem~\ref{thm:main-theorem} follows directly from Theorem~\ref{main-theorem-2}. For $\G = \SL_n$ with $p \mid n$ or $p=2,$ we note that $\mathrm{Ind}_{\mathbf {U}(\cO_\ell)}^{\GL_n(\cO_\ell)}(\theta) \cong \mathrm{Ind}_{\SL_n(\cO_\ell)}^{\GL_n(\cO_\ell)}\left( \mathrm{Ind}_{\mathbf {U}(\cO_\ell)}^{\SL_n(\cO_\ell)}(\theta)\right) $ and $\mathrm{Ind}_{\mathbf {U}(\cO_\ell)}^{\GL_n(\cO_\ell)}(\theta)$ is multiplicity free for non-degenerate $\theta$ by Theorem~\ref{main-theorem-2}. Therefore $\mathrm{Ind}_{\mathbf {U}(\cO_\ell)}^{\SL_n(\cO_\ell)}(\theta)$ must be multiplicity free for every non-degenerate character $\theta$ of $\cO_\ell.$

\section{Examples: $\GL_2$ and $\SL_2$} 
\label{sec:examples}

In this section, we discuss few details regarding the irreducible representations $\GL_2(\cO_\ell)$ and $\SL_2(\cO_\ell)$ for $\ell \geq 2$ to bring out the differences in these cases regarding Theorem~\ref{main-theorem-2}. 
Throughout this section, we assume that  $2 \nmid |\cO/\wp|$. 

In the construction of irreducible representations of $\G(\cO_\ell)$ for $\G = \GL_n$ or $\SL_n$ one usually focuses on the construction of its primitive irreducible representations. These are precisely the representations that can not be obtained via pull back from the representations of $\G(\cO_i)$ for $i < \ell$. Formally, the primitive representations of $\G(\cO_\ell)$ are defined as below. 
\begin{definition}
	An irreducible representation $\rho$ of $\G(\cO_\ell)$ is called {\rm primitive} if the orbit of its restriction to $\K_\ell^{\ell-1} $ does not contain a one dimensional representation $ \varphi_x $ where $x$ is a scalar matrix.
\end{definition}  It is to be noted that by definition every regular representation is primitive but converse need not to be true. However it turns out to be true for $\GL_2(\cO_\ell)$ and $\SL_2(\cO_\ell)$ due to the fact that any $x \in M_2(\mathbb F_q)$ is either a scalar matrix or a cyclic matrix. Therefore every one dimensional representation $\varphi_x$ of $\K_\ell^{\ell -1}$ is such that either $x$ is regular or scalar. In particular, an irreducible representation $\rho$ of $\GL_2(\cO_\ell)$ or $\SL_2(\cO_\ell)$ is regular if and only if it is primitive. Let $\mathcal{P}_{\GL_2(\cO_\ell)}$ and  $\mathcal{P}_{\SL_2(\cO_\ell)}$ be the set of inequivalent primitive  irreducible representations of $\GL_2(\cO_\ell)$ and $\SL_2(\cO_\ell)$ respectively. These sets are further partitioned into three types as follows. Any irreducible representation $\rho \in \mathcal{P}_{\GL_2(\cO_\ell)}$ or $\mathcal{P}_{\SL_2(\cO_\ell)}$ is called (of type) {\it cuspidal}, {\it split non-semisimple} and {\it split semisimple} if $\langle \rho|_{K_\ell^{\ell-1}} , \varphi_x   \rangle \neq 0$ implies characteristic polynomial of $x$ is irreducible, splits over $\mathbb F_q$ and has equal roots and splits over $\mathbb F_q$ with distinct roots  respectively.  
For the groups $\GL_2(\cO_\ell)$, the information  known regarding the dimensions of all elements of $\mathcal{P}_{\GL_2(\cO_\ell)}$, see \cite[Theorem~1.4]{MR2456275} is given in Table~\ref{GL2-reps}. 

\begin{table}[h]

\begin{center}

\begin{tabular}{c|c|c|c}
S.No.& Type of $\rho$  &  Number of $  \rho \in \mathcal{P}_{\GL_2(\cO_\ell)} $  & dimension of $\rho$  \\
\hline 
 \hline 
  
1. &  cuspidal & $\frac{1}{2} (q-1)(q^2-1) q^{2 \ell -3}$ & $q^{\ell-1}(q-1)$ \\
2. & split non-semisimple  & $(q-1) q^{2 \ell -2}$ & $(q^2-1)q^{\ell-2}$ \\
3. & split semisimple  & $\frac{1}{2} q^{2\ell-3} (q-1)^3 $ &  $q^{\ell-1} (q+1)$

\end{tabular}
\end{center}

\caption{ Dimensions of regular representations for  $\GL_2(\cO_\ell)$}
\label{GL2-reps}
\end{table} 

So in this case, we obtain the following.
 \[
 \sum_{\rho \in \mathcal{P}_{\GL_2(\cO_\ell) }} \dim(\rho) = (q^2 - 1) (q-1) q^{3 \ell -3} = \left|\frac{\GL_2(\cO_\ell)}{\mathbf U (\cO_\ell)}\right|.
 \] In case of $\GL_2(\mathbb F_q)$, it is well known that an irreducible representation of $\GL_2(\mathbb F_q)$ admits a $\theta$-Whittaker model for a non-degenerate $\theta$ if and only if it has dimension greater than one and this result by Theorem~\ref{main-theorem-2} and above computations generalize to $\GL_2(\cO_\ell)$ as follows.

\begin{theorem}
	
	\label{thm:primitive}  An irreducible representation of $\GL_2(\cO_\ell)$  for $\ell \geq 2$ admits a $\theta$-Whittaker model for any non-degenerate character $\theta$ of $\cO_\ell$ if and only if it is primitive. 
	
\end{theorem}

It is easy to see that the above result does not hold for $\GL_n(\cO_\ell)$ with $\ell \geq 2$ and $n \geq 3$. In comparison of the above result with $\SL_2(\cO_\ell)$, the numbers and dimensions of elements of $\mathcal{P}_{\SL_2(\cO_\ell)}$ are as given in Table~\ref{SL2-reps}, see \cite[Section~7]{MR2169043}.  
\begin{table}
\begin{center}
\begin{tabular}{c|c|c|c}
S.No.& Type of $\rho$  & Number of  $ \rho \in \mathcal{P}_{\SL_2(\cO_\ell)} $  & dimension of $\rho$ \\
\hline 
 \hline 
  
1. & cuspidal  & $\frac{1}{2} (q^2-1) q^{\ell -2}$ & $q^{\ell-1}(q-1)$ \\
2. & split non-semisimple & $4q^{\ell -1}$ & $\frac{1}{2}(q^2-1)q^{\ell-2}$ \\
3. & split semisimple & $\frac{1}{2} q^{\ell-2} (q-1)^2 $ &  $q^{\ell-1} (q+1)$

\end{tabular}
\end{center}
\caption{ Dimensions of regular representations for  $\SL_2(\cO_\ell)$}
\label{SL2-reps}
\end{table}
For group $\SL_2(\cO_\ell)$, we have the following.  
\[
\sum_{\rho \in \mathcal{P}_{\SL_2(\cO_\ell) }} \dim(\rho) = (q^2-1)(q+1)q^{2 \ell -3} \gneq (q^2-1) q^{2 \ell -4} = \left|\frac{\SL_2(\cO_\ell)}{\mathbf U (\cO_\ell)}\right|.
\] 
This in particular implies that for any non-degenerate character $\theta$ of $\mathbf U(\cO_\ell)$, there exists regular(primitive) representations $\rho_1$ and $\rho_2$ of $\SL_2(\cO_\ell)$ such that $\rho_1$ admits a $\theta$-Whittaker model but $\rho_2$ does not admit a $\theta$-Whittaker model. Further it is to be noted from the dimensions given in Tables~\ref{GL2-reps} and \ref{SL2-reps} that the cuspidal and split semisimple representations of $\GL_2(\cO_\ell)$ in fact restrict to irreducible representations of $\SL_2(\cO_\ell)$. This in particular implies the following. 
\begin{corollary}
Any cuspidal or split-semisimple regular representation of the group $\SL_2(\cO_\ell)$ admits a  $\theta$-Whittaker model for any non-degenerate character $\theta$ of $\mathbf U(\cO_\ell)$. 
\end{corollary}
This result also extends to $\SL_n(\cO_\ell)$ for $(p,2) = (p,n) = 1$ and we discuss this in  Section~\ref{sec:branching-rules}, see Corollary~\ref{cor:cuspidal-whittaker}.
\section{Restriction from $\GL_n(\cO_\ell)$ to $\SL_n(\cO_\ell)$ }
\label{sec:branching-rules}
In this section, we  prove Theorem~\ref{thm:multiplicity-free-induction}. We also obtain the branching rules for the restriction of regular representations of $\GL_n(\cO_\ell)$ to $\SL_n(\cO_\ell)$ for $(p,2) = (p,n) = 1$. 

\subsection{Proof of Theorem~\ref{thm:multiplicity-free-induction}} Let $\rho$ be a regular representation of $\GL_n(\cO_\ell)$ and let $\delta$ be a regular representation of $\SL_n(\cO_\ell)$ such that $\langle \delta,  \rho \rangle_{\SL_n(\cO_\ell )} \neq 0 $. We need to prove that $\langle \delta,  \rho \rangle_{\SL_n(\cO_\ell )} = 1 $. Let $\sigma$ be a regular representation of $\SL_n(\cO_\ell)$ such that $\langle \sigma , \mathrm{Ind}_{\mathbf U(\cO_\ell)}^{\SL_n(\cO_\ell)}(\theta) \rangle \neq 0$ and $\langle \rho , \mathrm{Ind}_{\SL_n(\cO_\ell)}^{\GL_n(\cO_\ell)}(\sigma) \rangle \neq 0$. The representation $\mathrm{Ind}_{\SL_n(\cO_\ell)}^{\GL_n(\cO_\ell)}(\sigma)$ is a subrepresentation of $\mathrm{Ind}_{\mathbf U(\cO_\ell)}^{\GL_n(\cO_\ell)}(\theta)$ and therefore is multiplicity free by Theorem~\ref{thm:main-theorem}. By the fact that $\langle \delta,  \rho \rangle_{\SL_n(\cO_\ell )} \neq 0 $, $\langle \sigma ,  \rho \rangle_{\SL_n(\cO_\ell )} \neq 0 $ and Clifford theory, we must have $\delta^g = \sigma $ for some $g \in \GL_n(\cO_\ell)$. This in particular implies that both  $\Ind_{\SL_n(\cO_\ell)}^{\GL_n(\cO_\ell)}(\delta )$ and $\Ind_{\SL_n(\cO_\ell)}^{\GL_n(\cO_\ell)}(\sigma )$ are isomorphic as $\GL_n(\cO_\ell)$-representations. Therefore $\Ind_{\SL_n(\cO_\ell)}^{\GL_n(\cO_\ell)}(\delta )$ is multiplicity free. By Frobenius reciprocity, we obtain $\langle \delta,  \rho \rangle_{\SL_n(\cO_\ell )} = 1 $.

\subsection{Branching rules}
For this  we recall the definition of the type of a regular representation of $\GL_n(\cO_\ell)$ and $\SL_n(\cO_\ell)$ as given in \cite{MR3737836}. Now onwards we will always assume $(p, 2) = (p, n) = 1$. 

\begin{definition} A matrix $\tau = (\tau_{d,e}) \in M_n(\mathbb Z_{\geq 0})$ is called $n$-typical if $n= \sum_{d, e} de\tau_{d,e}$. The set of $n$-typical matrices is denoted by $\mathcal{A}_n$. 
\end{definition}
Let $\mathcal{P}(\kk) $ be the set of all monic irreducible polynomials in $\kk[t]$. 
For any $x \in M_n(\kk)$, the characteristic polynomial of $x$ is of the form $\prod_i f_i^{e_i}$ for some finite list of elements $f_i \in \mathcal{P}(\kk) $ with $\sum_i e_i(\deg(f_i)) = n$. Therefore  the polynomial $\prod_i f_i^{e_i}$ determines an $n$-typical matrix $\mathbf{m}^x$ with $\mathbf{m}^x_{d,e} = |\{ i \mid \deg(f_i) = d \,\,\, \mathrm{and}\,\,\, e_i = e   \}|$.   It is well known that the regular elements of $M_n(\kk)$ up to conjugacy are determined uniquely by their characteristic polynomial. By above we have the following map. 
\[
\{ \mathrm{ Regular \,\, conjugacy \,\,  classes \,\, in \,\, } M_n(\kk) \} \to \mathcal{A}_n\,\,; x \mapsto \mathbf{m}^x
\]
We remark that for $q>n$, the map $x \to \mathbf{m}^x$ is a surjection. 
\begin{definition}($\tau$-regular element)
A regular element $x \in M_n(\kk)$ is called $\tau$-regular for $\tau \in \mathcal{A}_n$ if $\mathbf{m}^x = \tau$.
\end{definition}
\begin{definition} ($\tau$-regular representation)
A regular representation $\rho$ of $\G(\cO_\ell)$ is said to be $\tau$-regular if restriction of $\rho$ to $K_\ell^{\ell-1}$ contains $\varphi_x$ for a $\tau$-regular element $x$. 
\end{definition}
For $\tau \in \mathcal{A}_n$, $r \in \mathbb N$ and for a $\tau$-regular $x$, we define the following. 
\[
\begin{array}{ccl}
\mathbf {v} & := & |\GL_n(\kk)|, \\
\mathbf {u}^{\tau} &  := & |C_{\GL_n(\kk)}(x)|, \\
\upiota (\tau, r) &  := & gcd\left( \{  e \mid \exists\,\, d \,\,\mathrm{with}\,\, \tau_{d,e} \neq 0\} \cup \{r \} \right).
\end{array} 
\]

\begin{theorem}[\cite{MR3737836}, Theorems B and E]
\label{thm:regular-dimension}
For $\tau \in \mathcal{A}_n$, the dimension of any $\tau$-regular representation $\rho$ of $G(\cO_\ell)$ for $\ell \geq 2$ is given by the following. 
\[
\dim(\rho) =  \begin{cases}  q^{\binom{n}{2} (\ell-2) }\frac{\mathbf{v}}{\mathbf {u}^{\tau} }, &  \mathrm{for } \,\, \G(\cO_\ell) = \GL_n(\cO_\ell) \, \, \mathrm{with} \,\, (p,2) = 1,  \\ q^{\binom{n}{2} (\ell-2) }\frac{\mathbf{v}}{\upiota(\tau, q-1) \mathbf {u}^{\tau}}, & \mathrm{for} \,\,  \G(\cO_\ell) = \SL_n(\cO_\ell) \, \, \mathrm{with} \,\, (p,2) = (p,n) = 1.
\end{cases} 
\]
\end{theorem}
The proof of the above result is included in \cite[Thereoms~B, E]{MR3737836}. In the following remark we  clarify the differences that arise in the present formulation of above result and that of the corresponding statements in \cite{MR3737836}.
\begin{remark}
\begin{enumerate} 
 \item Our stated results are for $\G(\cO_\ell)$ and the results of \cite{MR3737836} are for $\G(\cO_{\ell+1})$. So $\ell$ is modified appropriately. 
 \item The statement of \cite[Theorem B]{MR3737836} that gives above result for $\G(\cO_\ell) = \GL_n(\cO_\ell)$, has the hypothesis $q>n$. However the remark after the statement of Theorem~B there clarifies that the above result is true even for $q \leq n$. The hypothesis $q > n$ is required there for some other counting arguments which we will never use, so we have omitted the condition $q > n$. 
 \item The proof of Theorem~\ref{thm:regular-dimension} for $\G(\cO_\ell) = \SL_n(\cO_\ell)$ is part of the proof of \cite[Theorem E]{MR3737836}. 
 \end{enumerate} 
\end{remark}

\begin{corollary}
\label{cor:dimensions-of-regular}
The dimension of any $\tau$-regular representation of $\SL_n(\cO_\ell)$ is obtained by dividing the dimension of a $\tau$-regular representation of $\GL_n(\cO_\ell)$ by $\upiota(\tau, q-1)$. 
\end{corollary}

\begin{definition}(special-regular representation of $\SL_n(\cO_\ell )$) A regular representation $\delta$ of $\SL_n(\cO_\ell)$ is called special regular if and only if $\delta \cong \mathrm{Res}^{\GL_n(\cO_\ell)}_{\SL_n(\cO_\ell)}(\rho)$ for a regular representation $\rho$ of $\GL_n(\cO_\ell)$. 
\end{definition}
By Corollary~\ref{cor:dimensions-of-regular} and the definition of special-regular, we obtain the following. 
\begin{lemma}
A $\tau$-regular representation of $\SL_n(\cO_\ell)$ is special-regular if and only if $\upiota(\tau, (q-1)) = 1$. 
\end{lemma}
 
\begin{theorem}
\label{prop:special-regular-characterization} A regular representation $\delta$ of $\SL_n(\cO_\ell)$ admits a  $\theta$-Whittaker model for every non-degenerate character $\theta$ of $\mathbf U(\cO_\ell)$ if and only if $\delta$ is special-regular. 
\end{theorem}
\begin{proof}
Let $\delta$ be a special-regular representations of $\SL_n(\cO_\ell)$, then $\delta \cong  \mathrm{Res}^{\GL_n(\cO_\ell)}_{\SL_n(\cO_\ell)}(\rho)$ for a  regular representation $\rho$ of $\GL_n(\cO_\ell)$. Let $\theta$ be a non-degenerate character of $\mathbf U(\cO_\ell)$. Then by Corollary~\ref{cor:hill-regular-iff-whittaker}, we have 
\[
 \langle  \delta , \mathrm{Ind}^{\SL_n(\cO_\ell)}_{\mathbf U(\cO_\ell)}\theta \rangle = \langle \mathrm{Res}^{\GL_n(\cO_\ell)}_{\SL_n(\cO_\ell)}(\rho), \mathrm{Ind}^{\SL_n(\cO_\ell)}_{\mathbf U(\cO_\ell)}\theta \rangle = \langle \rho, \mathrm{Ind}^{\GL_n(\cO_\ell)}_{\mathbf U(\cO_\ell)}\theta \rangle = 1. 
\]

Conversely, let $\delta$ be a regular representation of $\SL_n(\cO_\ell)$ such that $\delta$ is not a special-regular representation of $\SL_n(\cO_\ell)$ and admits a $\theta$-Whittaker model for every non-degenerate $\theta$. Then there exists $\delta'$ (possibly equivalent to $\delta$) such that both $\delta$ and $\delta'$ are regular representations of $\SL_n(\cO_\ell)$ and $\delta \oplus \delta'$ is a subrepresentation of $\mathrm{Res}^{\GL_n(\cO_\ell)}_{\SL_n(\cO_\ell)}(\rho)$ for some regular representation $\rho$ of $\GL_n(\cO_\ell)$. The representation $\delta'$ is a regular representation of $\SL_n(\cO_\ell)$, therefore by  Theorem~\ref{main-theorem-2} there exists a non-degenerate character $\theta$ such that $\delta'$ admits a $\theta$-Whittaker model. Then by our hypothesis, 
\[
\langle \rho, \mathrm{Ind}_{\mathbf U(\cO_\ell)} ^{\GL_n(\cO_\ell)} (\theta) \rangle \geq  \langle \delta \oplus \delta' , \mathrm{Ind}_{\mathbf U(\cO_\ell)} ^{\SL_n(\cO_\ell)}(\theta) \rangle = 2.
\]
This is a contradiction to Theorem~\ref{thm:main-theorem}. 
\end{proof}

Recall a regular element $x$ is called cuspidal if the characteristic polynomial of $x$ is irreducible and cuspidal representations of $\G(\cO_\ell)$ are the ones that lie above the cuspidal characters. 
A regular element $x$ is called split-semisimple if the characteristic polynomial splits over $\mathbf F_q$ and has distinct roots. Parallel to the cuspidal representations, the split semisimple representation of $\G(\cO_\ell)$ are defined. Both the cuspidal and split semisimple representations of $\SL_n(\cO_\ell)$ are special regular. This along with Theorem~\ref{prop:special-regular-characterization} directly gives the following. 
\begin{corollary}
\label{cor:cuspidal-whittaker}
Let $\rho$ be a cuspidal or a split semisimple representation of $\SL_n(\cO_\ell)$ for $(p,2) = (p,n) = 1$. Then $\rho$ admits a $\theta$-Whittaker model for every non-degenerate character $\theta$ of $\mathbf U(\cO_\ell)$. 
\end{corollary}

Next, we describe the branching rules for $\mathrm{Res}^{\GL_n(\cO_\ell)}_{\SL_n(\cO_\ell)}(\rho)$ for every regular representation $\rho$ of $\GL_n(\cO_\ell)$. 

\begin{theorem}
\label{thm:branching-numbers} For $(p, 2) = (p, n) = 1$. Let $\rho$ be a $\tau$-regular  representation of $\GL_n(\cO_\ell)$ for $\tau \in \mathcal{A}_n$. Then 
\[
\mathrm{End}_{\SL_n(\cO_\ell)} \left( \mathrm{Res}_{\SL_n(\cO_\ell)}^{\GL_n(\cO_\ell)} (\rho)  \right)  \cong \mathbb C^{ \upiota(\tau, q-1) } . 
\]
 \end{theorem}

 \begin{proof}
 By Theorem~\ref{thm:multiplicity-free-induction}, the representation $\mathrm{Res}_{\SL_n(\cO_\ell)}^{\GL_n(\cO_\ell)} (\rho)$ is multiplicity free. The group $\SL_n(\cO_\ell)$ is a normal subgroup of $\GL_n(\cO_\ell)$, therefore by Clifford theory  all the constituents of $\mathrm{Res}_{\SL_n(\cO_\ell)}^{\GL_n(\cO_\ell)} (\rho)$ must be of the same dimension and also must be $\tau$-regular representations of $\SL_n(\cO_\ell)$. Therefore by Corollary~\ref{cor:dimensions-of-regular}, the isomorphism of  $\mathrm{End}_{\SL_n(\cO_\ell)} \left( \mathrm{Res}_{\SL_n(\cO_\ell)}^{\GL_n(\cO_\ell)} (\rho)  \right) $ and $\mathbb C^{ \upiota(\tau, q-1) } $ follows. 
 \end{proof}
The following is a direct corollary of Theorem~\ref{thm:branching-numbers} by the fact that 
$\upiota(\tau, q-1) \leq n$ for every $\tau \in \mathcal{A}_n$.
 \begin{corollary}
 \label{cor:number-of-constituents}
For $(p, 2) = (p, n) = 1$. Let $\rho$ be a regular  representation of $\GL_n(\cO_\ell)$. Then $\mathrm{Res}_{\SL_n(\cO_\ell)}^{\GL_n(\cO_\ell)}(\rho)$ has at most $n$ irreducible constituents.
 \end{corollary}
\begin{remark} 
\begin{enumerate} 
\item We remark that Corollary~\ref{cor:number-of-constituents} can also be obtained by considering the index of the group  $Z(\GL_n(\cO_\ell))\SL_n(\cO_\ell)$ in $\GL_n(\cO_\ell)$, where $Z(\GL_n(\cO_\ell))$ denotes the centre of the group $\GL_n(\cO_\ell)$. 
\item Both of the above mentioned proofs of Corollary~\ref{cor:number-of-constituents}  very much depend on the condition $(p, n) = 1$. It is natural to ask whether this result continues to be true for general $p$ and $n$? 
\end{enumerate} 
\end{remark}

\noindent {\bf Acknowledgements: } Authors immensely thank Dipendra Prasad  for several helpful discussions regarding this paper and are also grateful to Sandeep Varma for his comments that helped in improving the current article. SPP acknowledges the support provided by Inspire Faculty Award [IFA17-MA102], DST, Government of India and PS acknowledges the support of UGC Centre for Advanced Study grant.

\bibliography{refs}{}
\bibliographystyle{siam}

\end{document}